\documentclass[12pt,twoside,reqno]{amsart}
\usepackage{pifont}
\usepackage{amsmath}
\usepackage{amsthm}
\usepackage{amsfonts}
\usepackage{amssymb}
\usepackage{latexsym}
\usepackage{amstext}
\usepackage{array}
\usepackage{graphicx}

\date{}
\allowdisplaybreaks[4] \footskip=15pt
\renewcommand{\uppercasenonmath}[1]{}

\newtheorem{thm}[subsection]{Theorem}
\newtheorem{cor}[subsection]{Corollary }
\newtheorem{Def}[subsection]{Definition}
\newtheorem{lem}[subsection]{Lemma}
\newtheorem{remark}[subsection]{Remark}

\newtheorem{prop}[subsection]{Proposition}
\newtheorem{exm}[subsection]{Example}
  
\newcommand{\bthm}{\begin{thm} }
\newcommand{\ethm}{\end{thm} }
\newcommand{\bpro}{\begin{prop}}
\newcommand{\epro}{\end{prop}}
\newcommand{\bdf}{\begin{Def}}
\newcommand{\edf}{\end{Def}}
\newcommand{\bexm}{\begin{exm}}
\newcommand{\eexm}{\end{exm}}
\newcommand{\blem}{\begin{lem}}
\newcommand{\elem}{\end{lem}}
\newcommand{\bpf}{\begin{proof}}
\newcommand{\epf}{\end{proof}}
\newcommand{\bcor}{\begin{cor}}
\newcommand{\ecor}{\end{cor}}
\newcommand{\ba}{\begin{array}}
\newcommand{\ea}{\end{array}}
\newcommand{\bea}{\begin{eqnarray}}
\newcommand{\eea}{\end{eqnarray}}

\newcommand{\brem}{\begin{remark}}
\newcommand{\erem}{\end{remark}}

\newsavebox{\tablebox}

\begin{document}

\begin{center}
{\large  \bf The graphs which are cospectral with the generalized pineapple graph}
\footnote {Supported by NSFC (Nos.  12071209, 12231009).}\\

 \vskip 0.8cm
 {\small  Borchen Li \ \ and\ \ Qingzhong Ji \footnote{Corresponding author.\\ \indent E-mail addresses: borchenli@smail.nju.edu.cn (B. Li),\; qingzhji@nju.edu.cn (Q. Ji)}}\\
{\small Department of Mathematics, Nanjing University, Nanjing
210093, P.R.China}

\vskip 3mm
\end{center}

{\bf Abstract:} {\small Let $p, k, q$ be positive integers with   $p-2 \geqslant k$ and let  $K_{p,k}^{q}$ be the generalized pineapple graph which is obtained by joining independent set of $q$ vertices with $k$ vertices of a complete graph $K_{p}.$ In \cite{TSH2}, Haemers et al. constructed graphs which cospectral with $K_{p,1}^{q}.$ In this paper, we determine all graphs which are cospectral with $K_{p,k}^{q}$ by considering the eigenvalues of its adjacency matrix. Moreover, We extend the conclusions of Haemers et al. to a broader context.}



\section{\bf Introduction}\label{1}

Throughout this paper, all graphs are simple and undirected. All concepts used in this paper can be found in \cite{BM} and in the articles  cited below, unless defined otherwise.
Let $G$ be a graph with vertex set $V(G)$ and edge set $E(G)$. The adjacency matrix of $G$ denoted by $A(G) = [a_{ij}],$ is a square matrix whose entries are indexed by $n \times n$ and $a_{ij} = 1$ if $\{i,j\} \in E(G)$ and $0$ otherwise, where $n$ is the order of $G$, i.e. the number of the vertices. The characteristic polynomial $f_{G}(x)$ of $G$ equals $det(xI-A(G)),$ where $I$ denotes the identity matrix. We also denote the all-one matrix by $J$ and all-zero matrix by $O.$ Since $A(G)$ is a real symmetric matrix, all roots of $f_{G}(x)$ are real numbers. Hence we can denote the eigenvalues of $A(G)$ by $\lambda_{1}(G) \geqslant \lambda_{2}(G) \geqslant \cdots \geqslant \lambda_{n}(G),$ where $n$ is the order of $G$. The eigenvalues of $A(G)$ compose the adjacency spectrum of $G.$ Two given graphs $G$ and $H$ are said to be cospectral if they share the same adjacency spectrum, denoted by $G \simeq H.$ A graph is determined by its adjacency spectrum, shortly DAS, if there is no other non-isomorphic graph with the same adjacency spectrum. Otherwise, it is said to be not determined by its adjacency spectrum, shortly non-DAS. We refer \cite{BH} \cite{CRS} \cite{GM} for more information about graph spectra.

The generalized pineapple graph $K_{p,k}^{q}\;(p-2 \geqslant k \geqslant 1,q \geqslant 1)$ is obtained by joining independent set of $q$ vertices with $k$ vertices of a complete graph $K_{p}.$  In an earlier paper \cite{TSH1}, Haemers et al. constructed graphs which cospectral with $K_{p,1}^{q}$ for every $p \geqslant 4$ and several values of $q$ such that they are non-isomorphic with $K_{p,1}^{q},$ and they showed that $K_{p,1}^{q}$ is DAS when $p = 3,$ $q \leqslant 2,$ or $p = q+1 = 4$ in the same paper. Recently, Haemers et al. \cite{TSH2} find for which values of $p$ and $q$ the pineapple graph $K_{p,1}^{q}$ is determined by its adjacency spectrum among connected graphs.

In this paper, we shall determined all graphs cospectral with the generalized pineapple graph $K_{p,k}^q.$ Recall that the complete split graph ${\rm CS}_{p,q}$ is obtained from the complete graph $K_{q}$ by joining independent set of $p$ vertices (in other words, ${\rm CS}_{p,q}$ is the complete multipartite graph $K_{p,1,\ldots,1}$ with $p+q$ vertices). Observe that $K_{p,p-1}^q=K_{p-1,p-1}^{q+1}={\rm CS}_{q+1,p-1}$ and $K_{p,p}^q={\rm CS}_{q,p}$ are the complete split graphs. Hence it suffices to consider the cases $p-2 \geqslant k \geqslant 1$ and $q\geqslant 1.$  We will prove that $K_{p,k}^{q}$ contains $0$ and $-1$ as eigenvalues with multiplicities $q-1$ and $p-2,$ respectively, and exactly three eigenvalues different from $-1$ and $0,$ see Lemma \ref{lem2.5}. What is more, two of them are positive and the other one is less than $-1.$ Let $\Omega$ be the set of graphs which eigenvalues at most three different from $-1$ and $0.$  Haemers determined $\Omega$ in paper \cite{Hae}. By use of this result we shall obtain the spectral characterization of the generalized pineapple graph $K_{p,k}^{q}.$ In \cite{TSH2}, the authors determined the all graphs which are cospectral with $K_{p,1}^{q}.$ Here we're going to determine  all graphs which are cospectral but nonisomorphic with $K_{p,k}^{q}$ in a similar way. It also generalized Haemers et al. conclusions to the generalized case.

\section{\bf Quotient graph and mixed extensions}\label{2}

In this section, we shall introduce briefly the quotient graph $G/\pi$ with respect to an equitable partition $\pi$ and mixed extensions of a graph. These are two common ways to make graphs smaller or bigger. We refer \cite{BH} \cite{Go} \cite{GR} for more information about the quotient graph and mixed extensions of graphs.

Let $G$ be a graph with the vertex set $V(G),$ and let $\pi= (V_{1},\ldots,V_{m})$ be a partition of $V(G)$. A partition $\pi$ of $V(G)$ is equitable partition if each vertex in $V_i$ has the same  number of neighbours  in $V_{j}$ for any $j\neq i.$  For any equitable partition $\pi= (V_{1},\ldots,V_{m})$ of $V(G),$ let $G/\pi$ be the quotient graph of $G$ with respect to the equitable partition $\pi.$ The quotient matrix $A(G/\pi)$ is the $m \times m$ matrix with $ij$-$th$ entry equal to  the number of the neighbors of a fixed vertex in $V_{i}$ in the part $V_{j}.$

\begin{lem}\label{2.1}{\rm(\cite{BH})}
Let $\pi$ be an equitable partition of the graph $G$ with $m$ parts. Then the characteristic polynomial of $A(G/\pi)$ divides the characteristic polynomial of $A(G).$
\end{lem}

The characteristic polynomial of ${\rm CS}_{p,q}$  is given as follows.

\begin{lem}\label{lem2.2}{\rm(\cite{TSH1})}
The characteristic polynomial of ${\rm CS}_{p,q}$ equals
$$f_{{\rm CS}_{p,q}}(x) = (x + 1)^{q-1}x^{p-1}(x^{2}-(q-1)x-pq).$$
\end{lem}

Let $G$ be a graph with the vertex set $V(G) = \{1, 2, \ldots, n\}$ and let $V_{1}, V_{2}, \ldots, V_{n}$ be mutually disjoint nonempty finite sets. We define a graph $H$ with vertex set $V(H) = V_{1} \bigcup V_{2}\bigcup \cdots \bigcup V_{n}$ as follows. For any $u \in V_{i}$ and $v \in V_{j}$ where $i, j \in V(G)$ and $i \neq j,$ $u$ is adjacent to $v$ in $H$ if and only if $i$ is adjacent to $j$ in $G,$ and for each $i \in V(G),$ all of the vertices of $V_{i}$ are either mutually adjacent, or mutually nonadjacent. We call such $H$ a mixed extension of $G.$ A mixed extension of a graph with $n$ vertices is represented by an $n$-tuple $(t_{1}, t_{2} \ldots, t_{n})$ of nonzero integers, where $t_{i} < 0$ indicates that $V_{i}$ is a coclique of order $-t_{i}$ and $t_{i} > 0$ means that $V_{i}$ is a clique of order $t_{i}$ \cite{Hae}. Hence, the generalized pineapple graph $K_{p,k}^{q}$ can be represented as a mixed extension of $P_{3}$ of type $(p-k,k,-q).$ We refer \cite{Hae} for more information about mixed extension of a graph.

In \cite{Hae}, Haemers determined the graph class $\Omega$ of all graphs (with no isolated vertices) which have exactly three eigenvalues different from $-1$ and $0.$ The results  were given by the following two lemmas.

\begin{lem}\label{lem2.3}{\rm(\cite{Hae})}
Suppose $H \in \Omega$ and has no isolated vertices. If $H$ is disconnected with exactly three eigenvalues different
from $-1$ or $0,$ then $H$ is one of the following graphs.
 \begin{itemize}
  \item[{\rm(1)}]$K_{p} \bigcup K_{q} \bigcup K_{r}$ with $p,q,r \geqslant 2.$
  \item[{\rm(2)}] $K_{p}\bigcup K_{q,r}$ with $p,q,r \geqslant 2.$
  \item[{\rm(3)}] $K_{p} \bigcup {\rm CS}_{q,r}$ with $p,q \geqslant 2,$ $r \geqslant 1.$
  \end{itemize}
\end{lem}

\begin{lem}\label{lem2.4}{\rm(\cite{Hae})}
Suppose $H \in \Omega$ and has no isolated vertices. If $H$  is connected with exactly three eigenvalues different
from $-1$ or $0,$ then $H$ is one of the following graphs.
  \begin{itemize}
  \item[{\rm(1)}]$H = K_{p,q,r}$ with $p,q,r \geqslant 2.$
  \item[{\rm(2)}] $H$ is a mixed extension of $K_{3}$ of type $(-p,-q,r)$ with $p,q \geqslant 2,$ $r \geqslant 1.$
  \item[{\rm(3)}] $H$ has exactly two positive eigenvalues and exactly one eigenvalue less than $-1.$
  \end{itemize}
\end{lem}

By doing simple calculation of the determinant, we obtain the  characteristic polynomials of $K_{p,k}^{q}$ as follows.

\begin{lem}\label{lem2.5} Let $p,k,q$ be positive integers with $p-2\geqslant k.$ Then the characteristic polynomial of $K_{p,k}^{q}$ as a product of coprime polynomials is given by
\begin{eqnarray}\label{lem2.5-1}f_{K_{p,k}^{q}}(x) = x^{q-1}(x+1)^{p-2}[x^{3}-(p-2)x^{2}-(p+kq-1)x+kq(p-k-1)].\end{eqnarray}
Hence, $K_{p,k}^{q}$ has exactly three eigenvalues different from $-1$ and $0$ and two of them are positive and the other one is less than $-1.$
\end{lem}

\begin{proof} \   It's easy to see that
the adjacency matrix of $K_{p,k}^{q}$ is
$$A(K_{p,k}^{q}) =
\begin{pmatrix}
    J_{k}-I_{k}&J_{k\times(p-k)} & J_{k\times q} \\
    J_{(p-k)\times k}& J_{p-k}-I_{p-k}&0_{(p-k)\times q}\\
    J_{q\times k}&0_{q\times(p-k)} & 0_{q} \\
\end{pmatrix}.$$
The given partition of $A(K_{p,k}^{q})$ is equitable with quotient matrix
$$Q=\begin{pmatrix}k-1&p-k&q\\k&p-k-1&0\\
k&0&0\end{pmatrix}.$$
Then  the  characteristic polynomial of $Q$ is $$f_Q(x)={\rm det}(xI_{p+q}-Q)=x^{3}-(p-2)x^{2}-(p+kq-1)x+kq(p-k-1).$$ Lemma 2.1 implies that $f_Q(x)$ is a divisor of $f_{K_{p,k}^{q}}(x).$ It's clear  that  $f_Q(-1)=qk(p-k)\neq 0$
and $f_Q(0)=kq(p-k-1)\neq 0.$ Hence $(x(x+1), f_Q(x))=1.$

Observe that $${\rm rank} A(K_{p,k}^{q}) =p+1,\; {\rm rank}\left(I_{p+q}+A(K_{p,k}^{q}) \right)=q+2.$$
Hence $f_{K_{p,k}^{q}}(x)$ has roots $0$ and $-1$  with multiplicities $p+q-{\rm rank} A(K_{p,k}^{q})=q-1$ and  $p+q-{\rm rank}\left(I_{p+q}+A(K_{p,k}^{q}) \right)=p-2$ respectively. Therefore the characteristic polynomial $f_{K_{p,k}^{q}}(x)$ of $K_{p,k}^{q}$ is given by

 $$f_{K_{p,k}^{q}}(x)=\left|xI-A(K_{p,k}^{q})\right|= x^{q-1}(x+1)^{p-2}[x^{3}-(p-2)x^{2}-(p+kq-1)x+kq(p-k-1)].$$
\qed\end{proof}

\

The class of graphs $\Omega$ was precisely described by Haemers in \cite{Hae} (ignoring isolated vertices) by the mixed extension. Let $\Omega^{*}$ denote the class of all connected graphs which have exactly two positive eigenvalues, one eigenvalue less than $-1$ and all other eigenvalues equal to either $-1$ or $0.$  It's clear that $\Omega^{*}$ is the subset of $\Omega.$ Notice that the generalized pineapple graph $K_{p,k}^{q}$ belongs to $\Omega^{*}$ and the determination of $\Omega^{*}$ could help us to find all graphs which are cospectral with $K_{p,k}^{q}.$ In \cite{Hae}, Haemers determined the class $\Omega^{*}$ by the following lemma.

\blem\label{lem2.6}{\rm(\cite{Hae})}
A graph $H$ belongs to $\Omega^{*}$ if and only if $H$ is one of the following.
  \begin{itemize}
  \item[{\rm(1)}] A mixed extension of $P_{3}$ of type $(-l,-m,n),$ $(-l,m,n),$ $(l,-m,n)$ or $(l,m,n)$ with $l,m \geqslant 1$ and $n \geqslant 2.$
  \item[{\rm(2)}] A mixed extension of $P_{4}$ of type $(l,-3,-2,-2),$ $(-2,m,n,-2)$ or $(l,-2,n,-3)$ with $l,m,n \geqslant 1.$
  \item[{\rm(3)}] A mixed extension of $P_{4}$ of type $(l,m,-n,s)$ with $n \geqslant 1$ and $(l,m,s)$ equals $(3,3,6),$ $(3,4,4),$ $(3,6,3),$ $(4,2,6),$ $(4,3,3),$ $(4,6,2),$ $(5,2,4),$ $(5,4,2),$ $(7,2,3),$ or $(7,3,2).$
  \item[{\rm(4)}] A mixed extension of $P_{4}$ of type $(l,m,n,s)$ with $(l,m,n,s)$ equal to $(2,2,2,7),$ $(2,2,3,4),$ $(2,2,6,3),$ $(2,3,2,5),$ $(2,3,4,3),$ $(2,5,2,4),$ $(2,5,3,3),$ or $(3,2,2,3).$
 \item[ {\rm(5)}] A mixed extension of $P_{5}$ of type $(1,l,-m,n,1)$ with $l,m,n \geqslant 1.$
 \end{itemize}
\elem

Since $K_{p,k}^{q}$ has exactly two positive eigenvalues, hence any graph $G$ with the same adjacency spectrum as $K_{p,k}^{q}$ must contain at most two components with at least one edge (shortly main component). Hence, we can determine the form of $G$ by the help of Lemmas \ref{lem2.3}, \ref{lem2.4} and \ref{lem2.6}. In the next section we will prove the main theorems by using the characteristic polynomials of graphs. For the characteristic polynomials of the graphs described in Lemma $\ref{lem2.6},$ we can get them by the use of the characteristic polynomial of the quotient matrix $A(G/\pi)$ with respect to the equitable partition $\pi,$ since all remaining eigenvalues are $0$ or $-1.$

\vskip 2mm

\section{\bf Main results}\label{3}

Let $K_{p,k}^{q}(p-2 \geqslant k \geqslant 1, q \geqslant 1)$ be a the generalized pineapple graph as introduced in section \S 2, and let $G$ be a graph with the same adjacency spectrum as $K_{p,k}^{q}.$ The graph $G$ has at most two components with at least one edge (shortly main components). First of all, we consider the case that $G$ has two main components with some isolated vertices.

\begin{thm}\label{thm3.1}
Let $K_{p,k}^{q}$ be a generalized pineapple graph with $p-2 \geqslant k \geqslant 1, q \geqslant 1,$
and let $G = G_{1} \bigcup G_{2} \bigcup a K_{1}$ be a graph which contains $a\geqslant 0$ isolated vertices and two main components $G_{1}$ and $G_{2}.$ Then $G$ is cospectral with $K_{p,k}^{q}$  if and only if
$G=K_{t }\bigcup {\rm CS}_{m,n} \bigcup aK_{1},$
where $a$ is a non-negative integer root of the equation \\
$kx^3 + [p - p^2 + (1 - q)k]x^2 + [(2q - 1)p^2 - 2pqk + (1 - 2q)p + 2qk^2 + (q - q^2)k]x + (k + 1 - q)p^2q + [(2q - 2)k -2k^2 + q - 1]pq + qk^3 + 2(1-q)qk^2 + (q^2 - 2q + 1)qk=0,$ \\
and $a<q-1,$ $p+q-k-1-a$ is a divisor of $(q-a)(p-k),$ furthermore,
$$\left\{\begin{array}{l}n=q-a,\\[5pt]
t=\dfrac{(q-a)(p-k)}{p+q-a-k-1},\\[10pt]
m=\dfrac{p^2-(k+1)p+kq-ka}{p+q-a-k-1}.\end{array}\right.$$
\end{thm}

\begin{proof} \ By Lemma \ref{lem2.5}, $K_{p,k}^q,$ and so $G,$ hence
$G_1\bigcup G_2$ has exactly three eigenvalues different from $-1$ and $0.$
By Lemma \ref{lem2.3}, $G_{1} \bigcup G_{2}$ is either of the form $K_{c}\bigcup K_{d,e}$ or $K_{t}\bigcup {\rm CS}_{m,n}$ for some $c, d, e,t,m\geqslant 2$ and $n\geqslant 1.$

{\bf Case (1).}  $G_{1} \bigcup G_{2}$ is of the form $K_{c}\bigcup K_{d,e}$ for some $c, d, e \geqslant 2.$

It is easy to see that the characteristic polynomial of $K_c\;\cup\; K_{d,e}\;\cup\; aK_1$ is
\begin{eqnarray}\label{thm3.1-1}\begin{array}{ll} f_{K_c\;\cup\; K_{d,e}\;\cup\; aK_1}(x) &= (x+1)^{c-1}[x-(c-1)]\cdot x^{d+e-2}(x^2-de)\cdot x^a \\ &=x^{a+d+e-2}(x+1)^{c-1}\left[x^{3}-(c-1)x^{2}-dex+(c-1)de\right].\end{array}\end{eqnarray}
 Put $g(x)=x^{3}-(c-1)x^{2}-dex+(c-1)de.$ Then $g(0)=(c-1)de\neq 0$ and $g(-1)=de(c-2)-c\neq 0$ since $c,d,e\geqslant 2.$ Hence $(x(x+1),g(x))=1.$
By comparing the coefficients and degrees of  polynomials \eqref{lem2.5-1} and \eqref{thm3.1-1}, we have
\begin{eqnarray}\label{3.1}\left\{\begin{array}{l}
      q-1 = a+d+e-2, \\
      p-2 = c-1, \\
      p+kq-1 = de, \\
      kq(p-k-1) = (c-1)de.
  \end{array}\right.\end{eqnarray}

From the first two equations of \eqref{3.1}, we obtain $q = a+d+e-1, p = c+1.$ Substitute $p, q$ into the last two equations of \eqref{3.1}, we have
\begin{eqnarray}\label{3.2}\left\{\begin{array}{l} k(a+d+e-1)+c = de,\\
      k(a+d+e-1)(c-k) = (c-1)de.
   \end{array}\right.\end{eqnarray}
Substitute the first equation  into the second equation of \eqref{3.2} to eliminate $de,$ we obtain $k(a+d+e-1)(k-1)+c^{2}-c = 0.$ It's impossible since $d,e,c\geqslant 2$, $k\geqslant 1$ and $a\geqslant 0.$

{\bf Case (2).} $G_{1} \bigcup G_{2}$ is of the form $K_{t}\bigcup {\rm CS}_{m,n}$ for some $t, m \geqslant 2$ and $n \geqslant 1.$

The characteristic polynomial $f_{{\rm CS}_{m,n}}(x)$ is given by a product of coprime polynomials  as follows
\begin{eqnarray}\label{thm3.1-2-1}f_{{\rm CS}_{m,n}}(x)=\left\{\begin{array}{ll}x^{n-1} (x+1)^{m-1}[x^2-(m-1)x-mn],&\text{if}\;n\geqslant 2,\\
 (x+1)^{m}(x-m),&\text{if}\;n=1.\end{array}\right.\end{eqnarray}
If $n=1,$ then
$$f_{G}(x) = f_{K_t\;\bigcup\;{\rm CS}_{m,n}\;\bigcup\;aK_1}(x) = x^{a}(x+1)^{m+t-1}[x^{2}-(m+t-1)x+(t-1)m]$$
has only two roots different from $0$ and $-1,$   thus $G$ is unable cospectral with $K_{p,k}^{q}.$ Hence, if $G = K_t\;\bigcup\;{\rm CS}_{m,n}\;\bigcup\;aK_1$ is cospectral with $K_{p,k}^{q},$ then  $n\geqslant 2$ and by  comparing the coefficients and degrees of  polynomials \eqref{lem2.5-1} and \eqref{thm3.1-2-1}, we have
\begin{align}\label{3.1-2}
   \begin{cases}
      q-1 = a+n-1, \\
      p-2 = m+t-2, \\
      p+kq-1 = mn-mt+m+t-1, \\
      kq(p-k-1) = (t-1)mn.
   \end{cases}
\end{align}
{\sl i.e.},
\begin{align}
   \begin{cases}\label{3.1-22}
      n =q- a, \\
       m+t=p, \\
       m(n-t)=kq, \\
      (p+q-a-k-1)t=(q-a)(p-k).
   \end{cases}
\end{align}
By calculating the system of equations \eqref{3.1-22}, we obtain that
$$\left\{\begin{array}{l}n=q-a,\\[5pt]
t=\dfrac{(q-a)(p-k)}{p+q-a-k-1},\\[10pt]
m=\dfrac{p^2-(k+1)p+kq-ka}{p+q-a-k-1}.\end{array}\right.  $$
and $a$ is a non-negative integer root of the equation \\
$kx^3 + [p - p^2 + (1 - q)k]x^2 + [(2q - 1)p^2 - 2pqk + (1 - 2q)p + 2qk^2 + (q - q^2)k]x + (k + 1 - q
)p^2q + [(2q - 2)k -2k^2 + q - 1]pq + qk^3 + 2(1-q)qk^2 + (q^2 - 2q + 1)qk=0.$
\qed\end{proof}

\begin{cor}\label{cor3.2}
Let $a \geqslant 3$ be an odd integer. Set $p=\frac{7a-1}{2},$ $q=\frac{5a-1}{2},$ $k=\frac{a-1}{2}.$ Then $K_{p,k}^{q}$ is non-DAS.

\begin{proof} Set  $t=a,$ $m=q=\frac{5a-1}{2},$ $n=\frac{3a-1}{2}$ are positive integers
and $G=K_{t}\cup CS_{m,n}\cup aK_{1}.$  By Theorem \ref{thm3.1}, it's  easy to see that $G$ is cospectral  with  $K_{p,k}^{q},$ $i.e.$ $f_{G}(x)= f_{K_{p,k}^{q}}(x).$ Hence  $K_{p,k}^{q}$ is non-DAS.
\qed \end{proof}
\end{cor}

Next, we consider the graphs $G$ which is cospectral with $K_{p,k}^{q}$ and has one main component and some isolated vertices.

\begin{thm}\label{thm3.3}
Let $K_{p,k}^{q}$ be a generalized pineapple graph with $p-2 \geqslant k \geqslant 1, q \geqslant 1,$
and let $G = G_{1} \bigcup a K_{1}$ be a graph which contains $a\geqslant 0$ isolated vertices and a main component $G_{1}.$ Then $G$ is cospectral with $K_{p,k}^{q}$ if and only if G is one of the following graphs for some $l, m, n \in \mathbb{Z}^{+}.$
\begin{itemize}
 \item[{\rm(1)}] $G_{1}$ is a mixed extension of $P_{3}$ of type $(-l,-m,n),$ where
    $$\left\{\begin{array}{l}
    l=\frac{kq(p-1)(p-k-1)}{k(k-1)q+(p-1)(p-2)},\\[6pt]
     m=\frac{k(k-1)q}{(p-1)(p-2)}+1, \\[6pt]
     n= p-1,
     \end{array}\right. $$
     and $a=q-\frac{k(k-1)q}{(p-1)(p-2)}-\frac{kq(p-1)(p-k-1)}{k(k-1)q+(p-1)(p-2)}.$ \\
 \item[{\rm(2)}] $G_{1}$ is a mixed extension of $P_{3}$ of type $(l,-m,n),$ where
    $$\left\{\begin{array}{l}
    l=\frac{1}{2}\left(p - \sqrt{p^{2}-p+2kq-\sqrt{(p+2kq)^{2}+8kpq(p-k-1)}}\right),\\[5pt]
    m=\frac{1}{4p}\left(p+2kq+\sqrt{(p+2kq)^{2}+8kpq(p-k-1)}\right), \\[5pt]
    n=\frac{1}{2}\left(p + \sqrt{p^{2}-p+2kq-\sqrt{(p+2kq)^{2}+8kpq(p-k-1)}}\right),
    \end{array}\right. $$
    and $a=\frac{1}{4p}\left(4pq-p-2kq-\sqrt{(p+2kq)^{2}+8kpq(p-k-1)}\right).$ \\
  \item[{\rm(3)}] $G_{1}$ is a mixed extension of $P_{3}$ of type $(l,m,n),$ where
    $$\left\{\begin{array}{l}
    l= \frac{1}{2}\left(p-m+1 - \sqrt{(p-m+1)^{2}-4(p-kq)}\right),\\[5pt]
    m= \frac{kq(p-k)}{p-kq}, \\[5pt]
    n= \frac{1}{2}\left(p-m+1 + \sqrt{(p-m+1)^{2}-4(p-kq)}\right),
     \end{array}\right. $$
    and $a= q-1.$
  \item[{\rm(4)}] $G_{1}$ is a mixed extension of $P_{4}$ of type $(l,-3,-2,-2),$ where $l =p-2,$ and $a =q-5.$ And, $k, p, q$ satisfy
    $$\left\{\begin{array}{l}
    kq = p+7, \\
    p^{2}-(k+6)p-7k+17=0.
     \end{array}\right. $$

 \item[{\rm(5)}] $G_{1}$ is a mixed extension of $P_{4}$ of type $(-2,m,n,-2),$ where
    $$\left\{\begin{array}{l}
    m= \frac{k}{2}\left(q + \sqrt{q}\right),\\
    n= \frac{k}{2}\left(q - \sqrt{q}\right),
     \end{array}\right. $$
    and $a= q-3.$ And, $k, p, q$ satisfy
    $$\left\{\begin{array}{l}
    kq = p-1,\\
     q = b^{2}, b \in \mathbb{Z}^{+} \setminus \{1\}.
     \end{array}\right. $$

  \item[{\rm(6)}] $G_{1}$ is a mixed extension of $P_{4}$ of type $(l,-2,n,-3),$ where
  $$\left\{\begin{array}{l}
    n= \frac{1}{18}\left[kq(p-k-1) +6(kq-p+1)\right], \\[10pt]
    l= \frac{3kq(p-k-1)}{kq(p-k+5)-6(p-1)},
    \end{array}\right. $$
   and $a= q-4.$ And, $k, p, q$ satisfy the equation \\ $k^{4}q^{2}-(2p+10)q^{2}k^{3}+[(p^{2}+10p+25)q^{2}+(30p-84)q]k^{2}-(30p^{2}+66p-96)kq+ \\
   144p^{2}+288p+144=0.$

  \item[{\rm(7)}] $G_{1}$ is a mixed extension of $P_{4}$ of type $(l,m,-n,s),$ where
   \begin{enumerate}
   \item[{\rm(\romannumeral1)}] $k =1,$  $a\in\{27, 63\}$ and $$(l,m,-n,s)=\left\{\begin{array}{lllll}(3,6,-20,3),&\text{if}\;a=63,&p=11, &k=1, &q=84,\\
   (4,3,-8,3),&\text{if}\;a=27,&p=9,&k=1,&q=36.\end{array}\right.$$
   \item[{\rm(\romannumeral2)}] $k=2,$ $a\in\{11,17,21,26,35,63\}$ and
   $$(l,m,-n,s)=\left\{\begin{array}{lllll}(3,4,-18,4),& \text{if} \;a=26, &p=10,&k=2,&q=45,\\
   (3,6,-48,3),& \text{if}\; a=63,&p=11,&k=2,&q=112,\\
   (4,3,-14,3),& \text{if} \;a=21,&p=9,&k=2,&q=36,\\
   (4,6,-27,2),& \text{if} \;a=35,&p=11,&k=2,&q=63,\\
   (5,2,-9,4),& \text{if}\; a=17,&p=10,&k=2,&q=27,\\
   (5,4,-8,2),& \text{if} \;a=11,&p=10,&k=2,&q=20.
   \end{array}\right.$$
   \item[{\rm(\romannumeral3)}] $k=3,$ $(l,m,-n,s)=(3,4,-63,4)$ and $a=56, p=10, k=3,q=120.$ \\
   \end{enumerate}
  \item[{\rm(8)}] $G_{1}$ is a mixed extension of $P_{5}$ of type $(1,l,-m,n,1)$\ $(l,n\geqslant 2),$ where
   $$\left\{\begin{array}{l}
    l= \frac{1}{2}\left(p-1 - \sqrt{(p-1)^{2}+2kq+p-1-s}\right),\\[5pt]
    m= \frac{1}{4(p-1)}(2kq-p+1+s), \\[5pt]
    n= \frac{1}{2}\left(p-1 + \sqrt{(p-1)^{2}+2kq+p-1-s}\right) , \\[5pt]
    s= \sqrt{(2kq-p+1)^{2}+8kq(p-1)(p-k)},
     \end{array}\right. $$
    and $a= \frac{1}{4(p-1)}(4pq-2(k+2)q-3(p-1)-s).$
\end{itemize}
\end{thm}

\begin{proof} \ It's clear that $G_{1} \in \Omega^{*}.$ Hence we need to consider all the cases of Lemma 2.6. Let $l, m, n \in \mathbb{Z}^{+}.$ Now, we can obtain the characteristic polynomials of the mixed extensions of $P_{3},$ $P_{4}$ and $P_{5}$ by the use of Lemma 2.1, and we next discuss them in the following cases, respectively.

\textbf{Case 1:} First of all, we examine the mixed extensions of $P_{3}.$

\textbf{(a)} Let $G_{1}$ be a mixed extension of $P_{3}$ of type $(-l,-m,n)$ with $l\geqslant 1,$ $m\geqslant 1,$ $n\geqslant 2.$   Then the characteristic polynomial of $G = G_{1} \bigcup a K_{1}$ is given by a product of coprime polynomials  as follows
\begin{eqnarray}\label{thm3.3-1-1} f_{G = G_{1} \bigcup a K_{1}}(x)=x^{a+l+m-2}(x+1)^{n-1}[x^{3}-(n-1)x^{2}-(lm+mn)x+lmn-lm].\end{eqnarray}
Hence, by \eqref{lem2.5-1} and \eqref{thm3.3-1-1},  we obtain that $G$ is cospectral with $K_{p,k}^{q}$ if and only if
\begin{align}
   \begin{cases}\label{3.3-1a1}
      q-1 = a+l+m-2, \\
      p-2 = n-1, \\
      p+kq-1 = lm+mn, \\
      kq(p-k-1) = lmn-lm,
   \end{cases}
\end{align}
{\sl i.e.},
\begin{align}
   \begin{cases}\label{3.3-1a2}
      m+l+a =q+1, \\
      n = p-1, \\
      lm+mn = p+kq-1, \\
      lmn-lm = kq(p-k-1).
   \end{cases}
\end{align}
By calculating the system of equations \eqref{3.3-1a2}, we obtain that
\begin{align*}
   \begin{cases}
    a=q-\frac{k(k-1)q}{(p-1)(p-2)}-\frac{kq(p-1)(p-k-1)}{k(k-1)q+(p-1)(p-2)}, \\[5pt]
    m=\frac{k(k-1)q}{(p-1)(p-2)}+1, \\[5pt]
    n= p-1, \\[5pt]
    l=\frac{kq(p-1)(p-k-1)}{k(k-1)q+(p-1)(p-2)}. \\[5pt]
   \end{cases}
\end{align*}

\textbf{(b)} Let $G_{1}$ be a mixed extension of $P_{3}$ of type $(l,-m,n)$ with  $l\geqslant 1,$ $m\geqslant 1,$ $n\geqslant 2.$
 Then the characteristic polynomial of $G = G_{1} \bigcup a K_{1}$ is given by a product of coprime polynomials  as follows
\begin{eqnarray}\label{thm3.3-1-2} \begin{array}{ll}f_{G }(x)=&x^{a+m-1}(x+1)^{l+n-2}
\cdot\left[x^{3}-(l+n-2)x^{2}\right.\\&\left.-(lm+mn-ln+l+n-1)x+2lmn-lm-mn\right].\end{array}
\end{eqnarray}
Hence,  by \eqref{lem2.5-1} and \eqref{thm3.3-1-2},  we obtain that $G$ is cospectral with $K_{p,k}^{q}$ if and only if
\begin{align}
   \begin{cases}\label{3.3-1b1}
      q-1 = a+m-1, \\
      p-2 = l+n-2, \\
      p+kq-1 = lm+mn-ln+l+n-1, \\
      kq(p-k-1) = 2lmn-lm-mn.
   \end{cases}
\end{align}
{\sl i.e.},
\begin{align}
   \begin{cases}\label{3.3-1b2}
      m+a =q , \\
      l+n =p, \\
      lm+mn-ln=kq, \\
      2lmn-lm-mn=kq(p-k-1).
   \end{cases}
\end{align}
Note that $l$ and $n$ are symmetrical, by calculating the system of equations \eqref{3.3-1b2},  we obtain that
\begin{align*}
   \begin{cases}
    a=\frac{1}{4p}\left(4pq-p-2kq-\sqrt{(p+2kq)^{2}+8kpq(p-k-1)}\right), \\[5pt]
    m=\frac{1}{4p}\left(p+2kq+\sqrt{(p+2kq)^{2}+8kpq(p-k-1)}\right), \\[7pt]
    n=\frac{1}{2}\left(p + \sqrt{p^{2}-p+2kq-\sqrt{(p+2kq)^{2}+8kpq(p-k-1)}}\right), \\[8pt]
    l=\frac{1}{2}\left(p - \sqrt{p^{2}-p+2kq-\sqrt{(p+2kq)^{2}+8kpq(p-k-1)}}\right). \\[8pt]
   \end{cases}
\end{align*}

\textbf{(c)} Let $G_{1}$ be a mixed extension of $P_{3}$ of type $(-l,m,n)$ with $l\geqslant 1,$ $m\geqslant 1,$ $n\geqslant 2.$   Then the characteristic polynomial of $G = G_{1} \bigcup a K_{1}$ is given by a product of coprime polynomials  as follows
\begin{eqnarray}\label{thm3.3-1-3} \begin{array}{ll}f_{G }(x)=& x^{a+l-1}(x+1)^{m+n-2}\\
&\cdot\left[x^{3}-(m+n-2)x^{2}-(lm+m+n-1)x+lmn-lm\right].\end{array}
\end{eqnarray}
Hence,  by \eqref{lem2.5-1} and \eqref{thm3.3-1-3},  we obtain that $G$ is cospectral with $K_{p,k}^{q}$ if and only if
\begin{align}
   \begin{cases}\label{3.3-1c1}
      q-1 = a+l-1, \\
      p-2 = m+n-2, \\
      p+kq-1 = lm+m+n-1, \\
      kq(p-k-1) = lmn-lm.
   \end{cases}
\end{align}
{\sl i.e.},
\begin{align}
   \begin{cases}\label{3.3-1c2}
      l+a = q , \\
      m+n = p, \\
      lm = kq, \\
      lmn-lm = kq(p-k-1).
   \end{cases}
\end{align}
By calculating the system of equations \eqref{3.3-1c2}, we get $m=k, l=q, a=0, n=p-k,$ and so $G\cong K_{p,k}^{q}.$

\textbf{(d)} Let $G_{1}$ be a mixed extension of $P_{3}$ of type $(l,m,n)$ with $l\geqslant 1,$ $m\geqslant 1,$ $n\geqslant 2.$   Then the characteristic polynomial of $G = G_{1} \bigcup a K_{1}$ is given by a product of coprime polynomials  as follows
\begin{eqnarray}\label{thm3.3-1-4} \begin{array}{ll}f_{G }(x)=&x^{a}(x+1)^{l+m+n-3}\cdot\left[x^{3}-(l+m+n-3)x^{2}\right.\\
&\left.-(2m+2n+2l-ln-3)x+lmn+ln-l-m-n+1\right].\end{array}
\end{eqnarray}
Hence,  by \eqref{lem2.5-1} and \eqref{thm3.3-1-4},  we obtain that $G$ is cospectral with $K_{p,k}^{q}$ if and only if
\begin{align}
   \begin{cases}\label{3.3-1d1}
     q-1 = a, \\
     p-2 = l+m+n-3, \\
     p+kq-1 = 2m+2n+2l-ln-3, \\
     kq(p-k-1) = lmn+ln-l-m-n+1.
   \end{cases}
\end{align}
{\sl i.e.},
\begin{align}
   \begin{cases}\label{3.3-1d2}
      a = q-1, \\
      l+m+n=p+1, \\
      ln=p-kq, \\
      lmn+ln=kq(p-k-1)+p.
   \end{cases}
\end{align}
By calculating the system of equations \eqref{3.3-1d2} and $l$ is symmetric with $n,$ we obtain that
\begin{align*}
   \begin{cases}
    a= q-1, \\
    m= \frac{kq(p-k)}{p-kq}, \\[5pt]
    n= \frac{1}{2}\left(p-m+1 + \sqrt{(p-m+1)^{2}-4(p-kq)}\right), \\[5pt]
    l= \frac{1}{2}\left(p-m+1 - \sqrt{(p-m+1)^{2}-4(p-kq)}\right).
   \end{cases}
\end{align*}

\textbf{Case 2:} Assume $G_{1}$ is the mixed extensions of $P_{4}.$

\textbf{(a)} Let $G_{1}$ be a mixed extension of $P_{4}$ of type $(l,-3,-2,-2)$ with  $l\geqslant 1.$   Then the characteristic polynomial of $G = G_{1} \bigcup a K_{1}$ is given by a product of coprime polynomials  as follows
\begin{eqnarray}\label{thm3.3-2-1} f_{G }(x)=x^{a+4}(x+1)^{l}[x^{3}-lx^{2}-(2l+10)x+12l]
\end{eqnarray}
Hence,  by \eqref{lem2.5-1} and \eqref{thm3.3-2-1},  we obtain that $G$ is cospectral with $K_{p,k}^{q}$ if and only if
\begin{align}
   \begin{cases}\label{3.3-2a1}
     q-1 = a+4, \\
     p-2 = l, \\
     p+kq-1 = 2l+10, \\
     kq(p-k-1) = 12l.
   \end{cases}
\end{align}
{\sl i.e.},
\begin{align}
   \begin{cases}\label{3.3-2a2}
      a =q-5 , \\
      l = p-2, \\
      l = kq-9, \\
      12l = kq(p-k-1).
   \end{cases}
\end{align}

By calculating the system of equations \eqref{3.3-2a2}, we obtain that
$$\left\{\begin{array}{l}
      a =q-5, \\
      l =p-2,\\
      kq =p+7,\\
      p^{2}-(k+6)p-7k+17=0.
\end{array}\right. $$

\textbf{(b)} Let $G_{1}$ be a mixed extension of $P_{4}$ of type $(-2,m,n,-2)$ with $m\geqslant 1, n\geqslant 1.$
Then the characteristic polynomial of $G = G_{1} \bigcup a K_{1}$ is given by a product of coprime polynomials  as follows
\begin{eqnarray}\label{thm3.3-2-2} f_{G }(x)=x^{a+2}(x+1)^{m+n-1}[x^{3}-(m+n-1)x^{2}-(2m+2n)x+4mn]
\end{eqnarray}
Hence,  by \eqref{lem2.5-1} and \eqref{thm3.3-2-2},  we obtain that $G$ is cospectral with $K_{p,k}^{q}$ if and only if
\begin{align}
   \begin{cases}\label{3.3-2b1}
     q-1 = a+2, \\
     p-2 = m+n-1, \\
     p+kq-1 = 2m+2n, \\
     kq(p-k-1) = 4mn,
   \end{cases}
\end{align}
{\sl i.e.},
\begin{align}
   \begin{cases}\label{3.3-2b2}
      a =q-3 , \\
      m+n=p-1, \\
      m+n=kq, \\
      4mn=kq(p-k-1).
   \end{cases}
\end{align}
Note that $m$ and $n$ are symmetrical, by calculating the system of equations \eqref{3.3-2b2}, we have
\begin{align*}
   \begin{cases}
      a= q-3, \\
      kq= p-1,\\
      m= \frac{k}{2}(q + \sqrt{q}),\\
      n= \frac{k}{2}(q - \sqrt{q}),
   \end{cases}
\end{align*}
where $q=b^2$ for some non-zero integer $b$ and $b\geqslant 2$ since $a=q-3\geqslant 0$.

\textbf{(c)} Let $G_{1}$ be a mixed extension of $P_{4}$ of type $(l,-2,n,-3)$ with $l\geqslant 1, n\geqslant 1.$  Then the characteristic polynomial of $G = G_{1} \bigcup a K_{1}$ is given by a product of coprime polynomials  as follows
\begin{eqnarray}\label{thm3.3-2-3} f_{G }(x)=x^{a+3}(x+1)^{l+n-1}[x^{3}-(l+n-1)x^{2}-(5n+2l-ln)x+6ln]
\end{eqnarray}
Hence,  by \eqref{lem2.5-1} and \eqref{thm3.3-2-3},  we obtain that $G$ is cospectral with $K_{p,k}^{q}$ if and only if
\begin{align}
   \begin{cases}\label{3.3-2c1}
     q-1 = a+3, \\
     p-2 = l+n-1, \\
     p+kq-1 = 5n+2l-ln, \\
     kq(p-k-1) = 6ln,
   \end{cases}
\end{align}
{\sl i.e.},
\begin{align}
   \begin{cases}\label{3.3-2c2}
      a =q- 4, \\
      l+n=p-1, \\
      3n-ln=kq-p+1, \\
      6ln=kq(p-k-1).
   \end{cases}
\end{align}
By calculating the system of equations \eqref{3.3-2c2}, we obtain that
\begin{align*}
   \begin{cases}
    a= q-4, \\
    n= \frac{1}{18}\left[kq(p-k-1) +6(kq-p+1)\right], \\[5pt]
    l= \frac{3kq(p-k-1)}{kq(p-k+5)-6(p-1)}.
   \end{cases}
\end{align*}  \\
and $k, p, q$ satisfy the equation \\
$$\begin{array}{l}k^{4}q^{2}-(2p+10)q^{2}k^{3}+[(p^{2}+10p+25)q^{2}+(30p-84)q]k^{2}\\-(30p^{2}+66p-96)kq+
   144p^{2}+288p+144=0.\end{array}$$

\textbf{(d)} Let $G_{1}$ be a mixed extension of $P_{4}$ of type $(l,m,-n,s)$ with $n \geqslant 1$ and $$(l,m,s) \in \left\{\begin{array}{l}(3,3,6), (3,4,4),(3,6,3),(4,2,6),(4,3,3),\\
(4,6,2),(5,2,4),(5,4,2),(7,2,3), (7,3,2)\end{array}\right\}.$$
For $i=1,2,\ldots,10,$ let $G_1^{(i)}$ be a mixed extension of $P_{4}$ of type $(l,m,-n,s)$ with $n \geqslant 1$ as follows:

\begin{lrbox}{\tablebox}
\begin{tabular}{|c|c|c|c|c|c|c|c|c|c|c|}
\multicolumn{11}{c}{}\\
\hline $(l,m,-n,s)$ &$(3,3,-n,6)$&$(3,4,-n,4)$&$(3,6,-n,3)$&$(4,2,-n,6)$&$(4,3,-n,3)$&$(4,6,-n,2)$&$(5,2,-n,4)$
&$(5,4,-n,2)$&$(7,2,-n,3)$&$(7,3,-n,2)$ \\
\hline
$G_1^{(i)}$ &$G_1^{(1)}$ &$G_1^{(2)}$ &$G_1^{(3)}$ &$G_1^{(4)}$ &$G_1^{(5)}$ &$G_1^{(6)}$ &$G_1^{(7)}$ &$G_1^{(8)}$ &$G_1^{(9)}$ &$G_1^{(10)}$\\
\hline

\multicolumn{11}{c}{}\\
\end{tabular}\end{lrbox}
\resizebox{5.8 in}{!}{\usebox{\tablebox}}
 Then the characteristic polynomial of $G^{(i)}=G_{1}^{(i)}\cup aK_{1}$  is given by a product of coprime polynomials  as follows
 $$f_{G^{(i)}}(x) =\left\{\begin{array}{ll}
 x^{a+n}(x+1)^{9} \left[x^{3}-9x^{2}-(8n-15)x+40n-25\right],&\text{if}\;i=1,\\
x^{a+n}(x+1)^{8} \left[ x^{3}-8x^{2}-(6n-9)x+34n+18\right],&\text{if}\;i=2,\\
x^{a+n}(x+1)^{9} \left[x^{3}-9x^{2}-(5n-6)x+37n+16\right],&\text{if}\;i=3,\\
x^{a+n}(x+1)^{9} \left[x^{3}-9x^{2}-(9n-15)x+45n+25\right],&\text{if}\;i=4,\\
x^{a+n}(x+1)^{7} \left[x^{3}-7x^{2}-(6n-4)x+30n+12\right],&\text{if}\;i=5,\\
x^{a+n}(x+1)^{9} \left[ x^{3}-9x^{2}-(5n+1)x+37n+9\right],&\text{if}\;i=6,\\
x^{a+n}(x+1)^{8} \left[x^{3}-8x^{2}-(8n-9)x+40n+18\right],&\text{if}\;i=7,\\
x^{a+n}(x+1)^{8} \left[x^{3}-8x^{2}-(6n+1)x+34n+8\right],&\text{if}\;i=8,\\
x^{a+n}(x+1)^{9} \left[x^{3}-9x^{2}-(9n-6)x+45n+16\right],&\text{if}\;i=9,\\
x^{a+n}(x+1)^{9} \left[x^{3}-9x^{2}-(8n+1)x+40n+9\right],&\text{if}\;i=10.\\
\end{array}\right.$$

Hence, for each $i = 1, 2,\ldots,10,$   $G^{(i)}$ is cospectral with $K_{p,k}^{q}$ if and only if
$$ f_{G^{(i)}}(x) = x^{q-1}(x+1)^{p-2}[x^{3}-(p-2)x^{2}-(p+kq-1)x+kq(p-k-1)].$$

(1) Then $G=G^{(1)}$ is cospectral with $K_{p,k}^{q}$ if and only if
$$\left\{\begin{array}{l}
a+n=q-1,\\
9=p-2,\\
8n-15=p+kq-1,\\
40n-25=kq(p-k-1).\end{array}\right.$$
This implies that $(5-k)kq=100,$ and we have $k=1, q=25$ or $k=4, q=25$. In either case, there is no integer solution for $n.$ Hence $G^{(1)}$ is not cospectral with $K_{p,k}^{q}.$  In an analogous way, we obtain that  $G^{(4)}$ and $G^{(10)}$ are not  cospectral with $K_{p,k}^{q}$  for any positive integers $n, p, k,q$ and $p-2\geqslant k.$

 (2) Then $G=G^{(2)}$  is cospectral with $K_{p,k}^{q}$ if and only if
$$\left\{\begin{array}{l}
a+n=q-1,\\
8=p-2,\\
6n-9=p+kq-1,\\
34n+18=kq(p-k-1),\end{array}\right.$$
Note that $1\leqslant k\leqslant 8=p-2,$ the above equation has only two solutions $(a,n,p,k,q)=(26,18,10,2,45),(56,63,10,3,120).$ Hence $G=G^{(2)}$  is cospectral with $K_{p,k}^{q}$ if and only if $G = G_{1}^{(2)} \bigcup 26K_{1} \simeq K_{10,2}^{45}$ where $G_{1}^{(2)}$ is a mixed extension of $P_{4}$ of type $(3,4,-18,4);$ or $G = G_{1}^{(2)} \bigcup 56 K_{1} \simeq K_{10,3}^{120},$ where $G_{1}^{(2)}$ is a mixed extension of $P_{4}$ of type $(3,4,-63,4).$

In a similar way, we deal with  other $G^{(i)}$ and obtain the results described in Theorem \ref{thm3.3} (7) which is also given  by the following table.

\begin{center}
\begin{lrbox}{\tablebox}
\begin{tabular}{|c|c|c|}
\multicolumn{3}{c}{}\\
\hline $(l,m,s)$ &$G^{(i)}=G_1^{(i)}\cup a K_1\simeq K_{p,k}^q$ &$G^{(i)}_{1}$ \text{is a mixed extension of} $P_{4}$ \text{of type} $(l,m,-n,s)$\\
\hline
$(3,3,6)$&\text{No}& \\
\hline
&$G_{1}^{(2)} \bigcup 26K_{1} \simeq K_{10,2}^{45}$&$(3,4,-18,4)$\\
\cline{2-3}
\raisebox{1.6ex}[0pt]{$(3,4,4)$}&$G_{1}^{(2)} \bigcup 56 K_{1} \simeq K_{10,3}^{120},$ &$(3,4,-63,4)$\\
\hline

&$G_{1}^{(2)}\bigcup 63K_{1} \simeq K_{11,1}^{84}$&$(3,6,-20,3)$\\
\cline{2-3}
\raisebox{1.6ex}[0pt]{$(3,6,3)$}&$G_{1}^{(3)} \bigcup 63K_{1} \simeq K_{11,2}^{112}$&$(3,6,-48,3)$\\

\hline
$(4,2,6)$&\text{No}&\\
\hline
&$G_{1}^{(5)} \bigcup 27K_{1} \simeq K_{9,1}^{36}$&$(4,3,-8,3)$\\
\cline{2-3}
\raisebox{1.6ex}[0pt]{$(4,3,3)$}&$G_{1}^{(5)} \bigcup 21K_{1} \simeq K_{9,2}^{36}$ &$(4,3,-14,3)$\\
\hline
$(4,6,2)$&$G_{1}^{(6)} \bigcup 35K_{1} \simeq K_{11,2}^{63}$ &$(4,6,-27,2)$\\
\hline
$(5,2,4)$&$G_{1}^{(7)} \bigcup 17K_{1} \simeq K_{10,2}^{27}$&$(5,2,-9,4)$\\
\hline
$(5,4,2)$&$G_{1}^{(8)} \bigcup 11K_{1} \simeq K_{10,2}^{20}$&$(5,4,-8,2)$\\
\hline
$(7,2,3)$ &\text{No}&\\
\hline
$(7,3,2)$ &\text{No}&\\
\hline

\multicolumn{3}{c}{}\\
\end{tabular}\end{lrbox}
\resizebox{3.8 in}{!}{\usebox{\tablebox}}
\end{center}

\textbf{(e)} Let $G_{1}$ be a mixed extension of $P_{4}$ of type $(l,m,n,s)$ with $(l,m,n,s) \in S,$ where
$$S=\left\{\begin{array}{l}(2,2,2,7), (2,2,3,4),(2,2,6,3),(2,3,2,5),\\(2,3,4,3),(2,5,2,4),(2,5,3,3), (3,2,2,3)\end{array}\right\}.$$
Let $G=G_1\cup a K_1.$ It's easy to show that $f_G(x)\neq f_{K_{p,k}^{q}}(x)$ for any $a\in\mathbb{N}$ and $p,k,q\in\mathbb{N}^*$ with $p-2\geqslant k.$ Hence $G$ is not cospectral with $K_{p,k}^q.$

\textbf{Case 3:} Finally, we consider the mixed extensions of $P_{5}.$ Let $G_{1}$ be a mixed extension of $P_{5}$ of type $(1,l,-m,n,1)$ with $l\geqslant 1,m\geqslant 1,n \geqslant 1.$
Then the characteristic polynomial of $G = G_{1} \bigcup a K_{1}$ is given by a product of coprime polynomials  as follows
\begin{eqnarray}\label{thm3.3-3-1} \begin{array}{ll}f_{G }(x)=&x^{a+m}(x+1)^{l+n-1}\\
&\cdot[x^{3}-(l+n-1)x^{2}-(mn-ln+lm+l+n)x+2lmn+ln].\end{array}
\end{eqnarray}
Hence,  by \eqref{lem2.5-1} and \eqref{thm3.3-3-1},  we obtain that $G$ is cospectral with $K_{p,k}^{q}$ if and only if
\begin{align}
   \begin{cases}\label{3.3-31}
     q-1 = a+m, \\
     p-2 = l+n-1, \\
     p+kq-1 = mn-ln+lm+l+n, \\
     kq(p-k-1) = 2lmn+ln,
   \end{cases}
\end{align}
{\sl i.e.},
\begin{align}
   \begin{cases}\label{3.3-32}
    a+m=q-1, \\
    l+n=p-1, \\
    lm+mn-ln=kq, \\
    2lmn+ln=kq(p-k-1).
   \end{cases}
\end{align}
Note that $n$ and $l$ are symmetrical, by calculating the system of equations \eqref{3.3-32},  we obtain that
\begin{align*}
   \begin{cases}
    a= \frac{4pq-2(k+2)q-3(p-1)-\sqrt{(2kq-p+1)^{2}+8kq(p-1)(p-k)}}{4(p-1)}, \\[5pt]
    m= \frac{2kq-p+1+\sqrt{(2kq-p+1)^{2}+8kq(p-1)(p-k)}}{4(p-1)}, \\[5pt]
    n= \frac{1}{2}(p-1 + \sqrt{(p-1)^{2}+2kq+p-1-\sqrt{(2kq-p+1)^{2}+8kq(p-1)(p-k)}}) , \\[5pt]
    l= \frac{1}{2}(p-1 - \sqrt{(p-1)^{2}+2kq+p-1-\sqrt{(2kq-p+1)^{2}+8kq(p-1)(p-k)}}).
   \end{cases}
\end{align*}

Let $s= \sqrt{(2kq-p+1)^{2}+8kq(p-1)(p-k)}.$ Then we obtain that
$$\left\{\begin{array}{l}
    l= \frac{1}{2}(p-1 - \sqrt{(p-1)^{2}+2kq+p-1-s}),\\[5pt]
    m= \frac{1}{4(p-1)}(2kq-p+1+s), \\[5pt]
    n= \frac{1}{2}(p-1 + \sqrt{(p-1)^{2}+2kq+p-1-s}) , \\[5pt]
    a= \frac{1}{4(p-1)}(4pq-2(k+2)q-3(p-1)-s).
     \end{array}\right. $$

\qed\end{proof}

\brem Given a generalized pineapple graph $K_{p,k}^{q}$ with $p-2\geqslant k\geqslant 1, q\geqslant 1.$ Theorems {\rm\ref{thm3.1}} and {\rm\ref{thm3.3}} determine all graphs which are cospectral with $K_{p,k}^{q}.$ Consequently, it is completely determined whether $K_{p,k}^{q}$ is DAS or non-DAS. \erem

\bexm For $p=8,$ $1\leqslant k\leqslant 6, 1\leqslant q\leqslant 100.$
There are $589$ generalized pineapple graphs $K_{p,k}^{q}$ which are DAS and the following $11$  generalized pineapple graphs $K_{p,k}^{q}$ which are non-DAS:
\begin{itemize}
\item[\rm(1)] $K_{8,2}^{84}\simeq G_1\cup 52K_1,$ where $G_1$ is a mixed extension of $P_3$ of type $(-28,-5,7);$
\item[\rm(2)] $K_{8,3}^{42}\simeq G_1\cup 24K_1,$ where $G_1$ is a mixed extension of $P_3$ of type $(-12,-7,7);$
\item[\rm(3)] $K_{8,3}^{91}\simeq G_1\cup 65K_1,$ where $G_1$ is a mixed extension of $P_3$ of type $(-13,-14,7);$
\item[\rm(4)] $K_{8,4}^{21}\simeq G_1\cup 9K_1,$ where $G_1$ is a mixed extension of $P_3$ of type $(-6,-7,7).$
\item[\rm(5)] $K_{8,1}^{16}\simeq G_1\cup 12K_1,$ where $G_1$ is a mixed extension of $P_3$ of type $(4,-4,4);$
\item[\rm(6)] $K_{8,3}^{4}\simeq G_1\cup K_1,$ where $G_1$ is a mixed extension of $P_3$ of type $(2,-3,6);$
\item[\rm(7)] $K_{8,3}^{11}\simeq G_1\cup 5K_1,$ where $G_1$ is a mixed extension of $P_3$ of type $(3,-6,5);$
\item[\rm(8)] $K_{8,3}^{16}\simeq G_1\cup 8K_1,$ where $G_1$ is a mixed extension of $P_3$ of type $(4,-8,4);$
\item[\rm(9)] $K_{8,2}^{6}\simeq G_1\cup 2K_1,$ where $G_1$ is a mixed extension of $P_4$ of type $(2,-2,5,-3);$
\item[\rm(10)] $K_{8,2}^{9}\simeq G_1\cup 4K_1,$ where $G_1$ is a mixed extension of $P_5$ of type $(1,2,-4,5,1);$
\item[\rm(11)] $K_{8,3}^{31}\simeq G_1\cup 15K_1,$ where $G_1$ is a mixed extension of $P_5$ of type $(1,3,-15,4,1).$

\end{itemize}
 \eexm

 \bexm For $p=11,$ $1\leqslant k\leqslant 9, 1\leqslant q\leqslant 100.$
There are $884$ generalized pineapple graphs $K_{p,k}^{q}$ which are DAS and the following $16$  generalized pineapple graphs $K_{p,k}^{q}$ which are non-DAS:
\begin{itemize}
\item[\rm(1)] $K_{11,1}^{56}\simeq K_8\cup CS_{3,36}\cup 20 K_1;$
\item[\rm(2)] $K_{11,2}^{45}\simeq G_1\cup 4K_1,$ where $G_1$ is a mixed extension of $P_3$ of type $(-40,-2,10);$
\item[\rm(3)] $K_{11,3}^{60}\simeq G_1\cup 28K_1,$ where $G_1$ is a mixed extension of $P_3$ of type $(-28,-5,10);$
\item[\rm(4)] $K_{11,3}^{90}\simeq G_1\cup 54K_1,$ where $G_1$ is a mixed extension of $P_3$ of type $(-30,-7,10);$
\item[\rm(5)] $K_{11,4}^{30}\simeq G_1\cup 10K_1,$ where $G_1$ is a mixed extension of $P_3$ of type $(-16,-5,10);$
\item[\rm(6)] $K_{11,5}^{18}\simeq G_1\cup 4K_1,$ where $G_1$ is a mixed extension of $P_3$ of type $(-10,-5,10);$
\item[\rm(7)] $K_{11,6}^{21}\simeq G_1\cup 7K_1,$ where $G_1$ is a mixed extension of $P_3$ of type $(-7,-8,10);$
\item[\rm(8)] $K_{11,4}^{15}\simeq G_1\cup 7K_1,$ where $G_1$ is a mixed extension of $P_3$ of type $(4,-8,7);$
\item[\rm(9)] $K_{11,5}^{3}\simeq G_1,$ where $G_1$ is a mixed extension of $P_3$ of type $(2,-3,9);$
\item[\rm(10)] $K_{11,5}^{49}\simeq G_1\cup 24K_1,$ where $G_1$ is a mixed extension of $P_3$ of type $(5,-25,6);$
\item[\rm(11)] $K_{11,2}^{7}\simeq G_1\cup 3K_1,$ where $G_1$ is a mixed extension of $P_5$ of type $(1,2,-3,8,1);$
\item[\rm(12)] $K_{11,3}^{13}\simeq G_1\cup 6K_1,$ where $G_1$ is a mixed extension of $P_5$ of type $(1,3,-6,7,1);$
\item[\rm(13)] $K_{11,3}^{25}\simeq G_1\cup 14K_1,$ where $G_1$ is a mixed extension of $P_5$ of type $(1,5,-10,5,1);$
\item[\rm(14)] $K_{11,4}^{6}\simeq G_1\cup K_1,$ where $G_1$ is a mixed extension of $P_5$ of type $(1,2,-4,8,1);$
\item[\rm(15)] $K_{11,4}^{29}\simeq G_1\cup 14K_1,$ where $G_1$ is a mixed extension of $P_5$ of type $(1,4,-14,6,1);$
\item[\rm(16)] $K_{11,6}^{14}\simeq G_1\cup 3K_1,$ where $G_1$ is a mixed extension of $P_5$ of type $(1,2,-10,8,1).$
\end{itemize}
 \eexm

\brem
Every situation in Theorems {\rm\ref{thm3.1}} and {\rm\ref{thm3.3}} can occur. For example, we have a table as follows for $k \leqslant 5.$

\begin{lrbox}{\tablebox}
\begin{tabular}{|c|c|c|l|}

\multicolumn{3}{c}{}\\

\hline
 $k$
 &\multicolumn{1}{|c|}{ \ \ \ \ \ \ \ \ \ \ \ \ $G = K_{t }\bigcup {\rm CS}_{m,n} \bigcup aK_{1}$ \ \ \ \ \ \ \ \  \ \ \ \ }
 &\multicolumn{1}{|c|}{$G = P_{3}$ of type $(-l,-m,n)$} \\
\hline

$1$
& $G = K_{8}\bigcup {\rm CS}_{3,36} \bigcup 20K_{1} \simeq K_{11,1}^{56}$
& $G = P_{3}$ of type $(-111,-1,183) \cong K_{184,1}^{111}$ \\
\hline

$2$
& $G = K_{8}\bigcup {\rm CS}_{9,16} \bigcup 16K_{1} \simeq K_{17,2}^{32}$
& $G = P_{3}$ of type $(-10,-3,5) \bigcup 8K_{1} \simeq K_{6,2}^{20}$ \\
\hline

$3$
& $G = K_{6}\bigcup {\rm CS}_{9,11} \bigcup 4K_{1} \simeq K_{15,3}^{15}$
& $G = P_{3}$ of type $(-6,-3,6) \bigcup 2K_{1} \simeq K_{7,3}^{10}$ \\
\hline

$4$
& $G = K_{9}\bigcup {\rm CS}_{8,27} \bigcup 3K_{1} \simeq K_{17,4}^{30}$
& $G = P_{3}$ of type $(-6,-7,7) \bigcup 9K_{1} \simeq K_{8,4}^{21}$ \\
\hline

$5$
& $G = K_{15}\bigcup {\rm CS}_{10,57} \bigcup 6K_{1} \simeq K_{25,5}^{63}$
& $G = P_{3}$ of type $(-10,-5,10) \bigcup 4K_{1} \simeq K_{11,5}^{18}$ \\
\hline

\end{tabular}
\end{lrbox}
\resizebox{5.5in}{!}{\usebox{\tablebox}}

\begin{lrbox}{\tablebox}
\begin{tabular}{|c|c|c|l|}

\hline
 $k$
 &\multicolumn{1}{|c|}{ \ \ \ \ \ \ \ \ \ \ $G = P_{3}$ of type $(l,-m,n)$ \ \ \ \ \ \ \ \ \ \ \ }
 &\multicolumn{1}{|c|}{ \ \ \ \ \ \ \ \ \ \ $G = P_{3}$ of type $(l,m,n)$ \ \ \ \ \ \ \ \ \ \ \ \ \ \ } \\
\hline

$1$
& $G = P_{3}$ of type $(2,-2,2) \bigcup 2K_{1} \simeq K_{4,1}^{4}$
& $G = P_{3}$ of type $(2,6,3) \bigcup 3K_{1} \simeq K_{10,1}^{4}$ \\
\hline

$2$
& $G = P_{3}$ of type $(2,-4,3) \bigcup 3K_{1} \simeq K_{5,2}^{7}$
& $G = P_{3}$ of type $(2,9,4) \bigcup 2K_{1} \simeq K_{14,2}^{3}$ \\
\hline

$3$
& $G = P_{3}$ of type $(2,-3,6) \bigcup K_{1} \simeq K_{8,3}^{4}$
& $G = P_{3}$ of type $(3,25,6) \bigcup 4K_{1} \simeq K_{33,3}^{5}$ \\
\hline

$4$
& $G = P_{3}$ of type $(4,-16,5) \bigcup 15K_{1} \simeq K_{9,4}^{31}$
& $G = P_{3}$ of type $(2,39,16) \bigcup 5K_{1} \simeq K_{56,4}^{6}$ \\
\hline

$5$
& $G = P_{3}$ of type $(2,-3,9) \simeq K_{11,5}^{3}$
& $G = P_{3}$ of type $(5,81,10) \bigcup 8K_{1} \simeq K_{95,5}^{9}$ \\
\hline

\end{tabular}
\end{lrbox}
\resizebox{5.5in}{!}{\usebox{\tablebox}}

\begin{lrbox}{\tablebox}
\begin{tabular}{|c|c|c|l|}

\hline
 $k$
 &\multicolumn{1}{|c|}{$G = P_{4}$ of type $(m,-3,-2,-2)$}
 &\multicolumn{1}{|c|}{ \ \ \ \ \ \ \ \ \ $G = P_{4}$ of type $(-2,m,n,-2)$ \ \ \ \ \ \ \ \ \ \ } \\
\hline

$1$
& $G = P_{4}$ of type $(3,-3,-2,-2) \bigcup 7K_{1} \simeq K_{5,1}^{12}$
& $G = P_{4}$ of type $(-2,3,1,-2) \bigcup K_{1} \simeq K_{5,1}^{4}$ \\
\hline

$2$
& $-$
& $G = P_{4}$ of type $(-2,6,2,-2) \bigcup K_{1} \simeq K_{9,2}^{4}$ \\
\hline

$3$
& $-$
& $G = P_{4}$ of type $(-2,9,3,-2) \bigcup K_{1} \simeq K_{13,3}^{4}$ \\
\hline

$4$
& $-$
& $G = P_{4}$ of type $(-2,12,4,-2) \bigcup K_{1} \simeq K_{17,4}^{4}$ \\
\hline

$5$
& $-$
& $G = P_{4}$ of type $(-2,15,5,-2) \bigcup K_{1} \simeq K_{21,5}^{4}$ \\
\hline

\end{tabular}
\end{lrbox}
\resizebox{5.5in}{!}{\usebox{\tablebox}}

\begin{lrbox}{\tablebox}
\begin{tabular}{|c|c|c|l|}

\hline
 $  k  $
 &\multicolumn{1}{|c|}{ \ \ \ \ \ \ \ \ \ $G = P_{4}$ of type $(n,-2,m,-3)$ \ \ \ \ \ \ \ \ \ \ }
 &\multicolumn{1}{|c|}{ \ \ \ \ \ \ \ \ \  $G = P_{5}$ of type $(1,l,-m,n,1)$ \ \ \ \ \ \ \ \  \ \ \ } \\
\hline

$1$
& $-$
& $G = P_{5}$ of type $(1,2,-4,2,1) \bigcup 7K_{1} \simeq K_{5,1}^{12}$ \\
\hline

$2$
& $G = P_{4}$ of type $(2,-2,5,-3) \bigcup 2K_{1} \simeq K_{8,2}^{6}$
& $G = P_{5}$ of type $(1,2,-8,3,1) \bigcup 8K_{1} \simeq K_{6,2}^{17}$  \\
\hline

$3$
& $G = P_{4}$ of type $(3,-2,15,-3) \bigcup 2K_{1} \simeq K_{19,3}^{6}$
& $G = P_{5}$ of type $(1,3,-15,4,1) \bigcup 15K_{1} \simeq K_{8,3}^{31}$  \\
\hline

$4$
& $-$
& $G = P_{5}$ of type $(1,4,-24,5,1) \bigcup 24K_{1} \simeq K_{10,4}^{49}$ \\
\hline

$5$
& $-$
& $G = P_{5}$ of type $(1,5,-35,6,1) \bigcup 35K_{1} \simeq K_{12,5}^{71}$ \\
\hline

\end{tabular}
\end{lrbox}
\resizebox{5.5in}{!}{\usebox{\tablebox}}

\erem

\

\

\end{document}